\documentclass{amsart}
\usepackage{amssymb, amsmath, latexsym, mathabx, dsfont,tikz}
\usepackage{multirow}
\usepackage{setspace}
\usepackage{enumerate}
\usepackage{diagbox}
\usepackage{stmaryrd}
\usepackage{makecell}

\renewcommand{\baselinestretch}{\baselinestretch}
\renewcommand{\baselinestretch}{1.1}
\numberwithin{equation}{section}

\newtheorem{thm}{Theorem}[section]
\newtheorem{lem}[thm]{Lemma}

\theoremstyle{definition}

\theoremstyle{remark}
\newtheorem{rmk}[thm]{Remark}

\numberwithin{equation}{section}

\newcommand{\ra}{{\ \longrightarrow \ }}
\newcommand{\nra}{{\ \longarrownot\longrightarrow \ }}

\newcommand{\gen}{\text{gen}}
\newcommand{\ord}{\text{ord}}

\newcommand{\n}{{\mathbb N}}
\newcommand{\z}{{\mathbb Z}}
\newcommand{\q}{{\mathbb Q}}
\newcommand{\Mod}[1]{\ (\mathrm{mod}\ #1 )}
\newcommand{\zpd}{\overset{\hbox{\tiny ${}_\bullet$}}{\z}_p}
\newcommand{\zpds}{(\overset{\hbox{\tiny ${}_\bullet$}}{\z}_p)^2}
\newcommand{\qpd}{\overset{\hbox{\tiny ${}_\bullet$}}{\q}_p}
\newcommand{\qpds}{(\overset{\hbox{\tiny ${}_\bullet$}}{\q}_p)^2}
\newcommand{\qtwods}{(\overset{\hbox{\tiny ${}_\bullet$}}{\q}_2)^2}
\newcommand{\ani}{{\mathbb A}}
\newcommand{\hyper}{{\mathbb H}}
\newcommand{\both}{{\mathbb B}}

\begin{document}
\title{The rank of new regular quadratic forms}

\author{Mingyu Kim}
\address{Department of Mathematics, Sungkyunkwan University, Suwon 16419, Korea}
\email{kmg2562@skku.edu}
\thanks{This work of the first author was supported by the National Research Foundation of Korea(NRF) grant funded by the Korea government(MSIT) (NRF-2021R1C1C2010133).}

\author{Byeong-Kweon Oh}
\address{Department of Mathematical Sciences and Research Institute of Mathematics, Seoul National University, Seoul 151-747, Korea}
\email{bkoh@snu.ac.kr}
\thanks{This work of the second author was supported by the National Research Foundation of Korea(NRF) grant funded by the Korea government(MSIT) (NRF-2019R1A2C1086347 and NRF-2020R1A5A1016126).}

\begin{abstract}
A (positive definite and integral) quadratic form $f$ is called regular if it represents all  integers that are locally represented. It is known that there are only 
finitely many regular ternary quadratic forms up to isometry. However, there  are infinitely many equivalence classes of regular quadratic forms of rank $n$ for any integer $n$ greater than or equal to $4$. 
A regular quadratic form  $f$ is called {\it new} if there does not exist a proper subform $g$ of $f$ such that the set of integers that are represented by $g$  is equal to the set of integers that are represented  by $f$.  In this article, we prove that the rank of any new regular quadratic form is bounded by an absolute constant.  
\end{abstract}
\maketitle

\section{Introduction}

 For a positive definite integral quadratic form 
$$
f(x_1,x_2,\dots,x_n)=\sum_{i,j=1}^n a_{ij}x_ix_j \qquad (a_{ij}=a_{ji} \in \z),
$$
 the set of integers that are represented by $f$ is denoted by $Q(f)$.  It is quite an old number theory problem to determine the set $Q(f)$ precisely. In fact, there is no known general method at present if the number of variables of $f$, which we call the rank of $f$, is three.   
 We say $f$ is {\it regular} if it represents all integers that are locally represented. Since we may effectively determine all integers that are locally represented by $f$ by \cite{OM2}, it is easy to determine the set $Q(f)$ 
 precisely if $f$ is regular.  After the terminology ``regularity'' was first introduced in 1927 by Dickson, 
Watson proved in his unpublished thesis \cite{w0} that there are only finitely many equivalence classes of (primitive) regular (positive definite integral) ternary quadratic forms. In 1995, Earnest showed in \cite{ea} that there are infinitely many regular quaternary quadratic forms (see also \cite{kimb}).  
Since the diagonal quadratic form $x^2+y^2+z^2+t^2$ represents all nonnegative integers, the quadratic form $I(f)=x^2+y^2+z^2+t^2+f(z_1,z_2,\dots,z_k)$ of rank $k+4$ is also regular for any quadratic form $f(z_1,z_2,\dots,z_k)$ of rank $k$. One may naturally consider the terminology of a ``new'' regular quadratic form to avoid this kind of an obvious example.  We say a regular quadratic form $f$ is new if there is no proper subform $g$ of $f$ such that $Q(g)=Q(f)$.  From the definition, the quadratic form $I(f)$ defined above is not new, though it is regular. Now, one may ask whether or not there are finitely many new regular quadratic forms of given rank. Related with this question, one may show that   the quaternary quadratic form $I_r=x^2+y^2+z^2+2^{2r-1}t^2$ is new regular for any positive integer $r$.  The regularities of these forms can partly  be  explained by the regularity of the ternary subform $x^2+y^2+z^2$.  A quadratic form $f$ is {\it strongly new regular} if it is a new regular quadratic form of rank greater than or equal to $4$ and, moreover,  there is no ternary subform $h$ of $f$ such that 
\begin{equation} \label{1}
Q(h)=Q(f)\cap Q(\q h),
\end{equation}
where $Q(\q h)$ is the set of integers that are represented by $h$ over $\q$.  Note that the quaternary quadratic form $I_r$ is not strongly new regular, though it is new regular for any positive integer $r$.   
 It was proved in \cite{CEO} that there are only finitely many equivalence classes of  strongly new regular quadratic forms(in fact, the terminology ``new'' in \cite{CEO} implies ``strongly new'' in the article).  To prove this, the authors heavily use the condition \eqref{1} to show that the $4$-th successive minimum of any strongly new regular quadratic form is bounded by an absolute constant.  This result implies that there are only finitely many strongly new regular quadratic forms of given rank, and furthermore, the rank of strongly new regular quadratic forms is bounded by an absolute constant. 

The aim of this article is to prove that the rank of new regular quadratic forms is bounded by an absolute constant, though there are infinitely many equivalence classes of new regular quaternary quadratic forms.  

Throughout this paper, the geometric language  of quadratic spaces and lattices will be adopted following \cite{OM}.   Let $R$ be the ring of rational integers $\z$ or the ring of $p$-adic integers $\z_p$ for a prime $p$.
An $R$-lattice $L=R\mathbf{x}_1+R\mathbf{x}_2+\cdots+R\mathbf{x}_k$ is a free $R$-module equipped with a non-degenerate symmetric bilinear map $B : L\times L \to R$.  The corresponding Gram matrix $M_L$ with respect to the basis $\{\mathbf x_i\}_{i=1}^k$ is defined by $(B(\mathbf{x}_i,\mathbf{x}_j))_{1\le i,j\le k}$. 
We always assume that every $R$-lattice $L$ is {\it integral} in the sense that the scale $\mathfrak{s}L$ of $L$ is contained in  $R$ and is positive definite when $R=\z$, that is, the corresponding Gram matrix  $M_L$ is positive definite.  For an $R$-lattice $L=R\mathbf{x}_1+R\mathbf{x}_2+\cdots+R\mathbf{x}_k$ and a symmetric $k\times k$ matrix $M=(m_{ij})$, if $B(\mathbf x_i,\mathbf x_j)=m_{ij}$ for any $i,j$ with $1\le i,j\le k$, 
then  we write 
$$
L \simeq M.
$$ 
In particular, if $M$ is diagonal, then we write $L \simeq \langle m_{11},m_{22},\dots,m_{kk}\rangle$. 

Let $L$ be a $\z$-lattice of rank $k$. We say $L$ is primitive if $\mathfrak s(L)=\z$. Here the scale $\mathfrak s(L)$ of $L$ is defined by the ideal generated by $\{B(\mathbf x,\mathbf y): \mathbf x,\mathbf y \in L\}$ in $\z$.  The norm $\mathfrak n(L)$ of $L$  is defined by the ideal generated by $\{Q(\mathbf x) : \mathbf x \in L\}$ in $\z$. 
The discriminant $dL$ of $L$ is defined by the determinant of the Gram matrix of $L$.
For any prime $p$, we define $L_p=L\otimes \z_p$, which is considered as a $\z_p$-lattice. For any positive integer $n \le k$, the set of all $\z$-lattices of rank $n$ that are represented by $L$ is denoted by $Q_n(L)$.  We also denote the set of $\z_p$-lattices of rank $n$ that are represented by $L_p$ by $Q_n(L_p)$. Finally, we denote by  $Q_n(\gen(L))$ the set of $\z$-lattices of rank $n$ that are represented by $L_p$ over $\z_p$ for any prime $p$. 
It is well known that $Q_n(\gen(L))$ is exactly equal to the set of $\z$-lattices of rank $n$ that are represented by a $\z$-lattice in the genus of $L$.   

A $\z$-lattice $L$ is called $n$-regular if $Q_n(L)=Q_n(\gen(L))$. There are  various finiteness results of equivalence classes of $n$-regular lattices of given rank, for example, \cite{ea}, \cite{jks}, \cite{o1}, and \cite{w0} for the case when $n=1$,  and \cite{CEO},  \cite{co}, and  \cite{ea1}  for the case when $n\ge 2$. 
  
An $n$-regular $\z$-lattice $L$ is called new if there is no proper sublattice $M$ of $L$ such that $Q_n(M)=Q_n(L)$. It was proved in \cite{CEO} that there are only finitely many new $n$-regular $\z$-lattices for any $n \ge 2$.  However, there are infinitely many equivalence classes of new $1$-regular $\z$-lattices of rank $4$, as mentioned above.  For this exceptional case, it is also proved in \cite{CEO} that there are only finitely many equivalence classes of strongly new $1$-regular $\z$-lattices.  In this article, we prove that the rank of new regular $\z$-lattices is bounded by an absolute constant. Our main ingredient is to use Theorem 3.3 of \cite{KKO}, which is as follows: Let $\mathcal S$ be a set of $\z$-lattices with bounded rank. A $\z$-lattice $L$ is called $\mathcal S$-universal if it represents all $\z$-lattices in $\mathcal S$. Then there always is a finite subset $\mathcal S_0 \subset S$ such that any $\mathcal S_0$-universal $\z$-lattice is, in fact, $\mathcal S$-universal. Such a finite set $\mathcal S_0$ is called a {\it finite $\mathcal S$-universality criterion set}. To prove the main result, we will critically use this theorem.  Since there is no known  quantitative version of Theorem 3.3 of \cite{KKO}, any explicit upper bound of the rank of new regular $\z$-lattices will not be provided.   

Throughout this article,  the anisotropic  even unimodular binary $\z_2$-lattice will be denoted by $\ani$, and the isotropic even unimodular binary $\z_2$-lattice will be denoted by $\hyper$. Hence we have
$$
\ani\simeq \begin{pmatrix} 2&1\\1&2\end{pmatrix} \quad \text{and} \quad \hyper\simeq \begin{pmatrix} 0&1\\1&0\end{pmatrix},
$$ 
for some suitable bases for $\ani$ and $\hyper$.

Any unexplained notations and terminologies can be found in  \cite {ki} or \cite {OM}.

\section{The set of integers represented by a $\z_p$-lattice}

Let $p$ be a prime. Throughout this article, we use the following notations:
$$
\qpd=\q_p-\{0\},\ \   \qpds=\{ \gamma^2 : \gamma \in \overset{\hbox{\tiny ${}_\bullet$}}{\q}_p \},  \quad 
\zpd=\z_p-\{0\},  \ \  \zpds=\{ \gamma^2 : \gamma \in \zpd \},
$$
and
$$
\z_p^{\times}=\z_p-p\z_p,  \quad \quad  (\z_p^{\times})^2= \{ \gamma^2 : \gamma \in \z_p^{\times} \}.
$$
We define a subset $SC_p$ of $\z_p$ by
$$
SC_p=\begin{cases} \{1,3,5,7,2\cdot 1,2\cdot 3,2\cdot 5,2\cdot 7\} & \text{if}\ \ p=2,\\
\{ 1,\Delta_p,p,p\Delta_p\} & \text{otherwise,}\end{cases}
$$
where $\Delta_p$ denotes a non square unit in $\z_p^{\times}$ for each odd prime $p$.  

Let  $L$ be a $\z_p$-lattice of rank greater than or equal to $4$.
Let  $s$ be any element in $SC_p$.
Since we are assuming that $\text{rank}(L) \ge 4$, there is a nonnegative integer $u$ such that
$$
sp^{2u}\ra L, \quad \text{whereas} \quad sp^{2u-2}\nra L.
$$
With this integer $u$,  we define 
$$
\eta_{p,s}(L)=sp^{2u} \quad \text{and} \quad \nu_{p,s}(L)=\ord_p(\eta_{p,s}(L)).
$$
We also define
$$
\nu_p(L)=\max \{ \nu_{p,s}(L) : s\in SC_p\} \ \  \text{and} \ \ 
\nu'_p(L)=\max \left\{ \nu_{p,s}(L) : s\in SC_p-\{s_0\} \right\},
$$
where $s_0$ is any element in $SC_p$ satisfying $\nu_p(L)=\nu_{p,s_0}(L)$.

Let $K$ be a ternary $\z_p$-lattice.
If $K$ is isotropic, then the ternary quadratic space $\q_pK$ is universal and thus we may define the above quantities $\eta_{p,s}(K)$, $\nu_{p,s}(K)$, $\nu_p(K)$, and $\nu'_p(K)$ in the same manners.
However, if $K$ is anisotropic, then $\eta_{p,s}(K)$ cannot be defined for the element $s_1\in SC_p$ satisfying $-s_1\cdot dK\in \qpds$.
Though we do not define $\eta_{p,s_1}(K)$ in this case, we define $\nu_{p,s_1}(K)=\infty$ and $\nu_p(K)=\infty$.
Hence, $\nu_p(K)=\infty$ if and only if $K$ is anisotropic.
For any other $s\in SC_p-\{s_1\}$, we define both  $\eta_{p,s}(K)$ and $\nu_{p,s}(K)$ as above so that $\nu'_p(K)$ is defined by
$$
\nu'_p(K)=\max \left\{ \nu_{p,s}(K) : s\in SC_p-\{s_1\} \right \}.
$$
It is clear that $\nu_p(L)=1$ if and only if $Q(L)=\z_p$.
For a prime $p$ and a $\z$-lattice $L$ with rank greater than or equal to $3$, we define $\nu_p(L)$ and $\nu'_p(L)$ by
$$
\nu_p(L)=\nu_p(L_p)\ \ \text{and}\ \ \nu'_p(L)=\nu'_p(L_p).
$$

\begin{lem} \label{lemnunuprime}
Let $L$ be a $\z_p$-lattice with rank greater than or equal to $4$. Then we have
\begin{itemize} 
\item [(i)] $\nu_p(L)=\min \{ r : \vert \{ s\in SC_p: p^{r-1}s\ra L_p \} \vert =\vert SC_p\vert \}$,
\item [(ii)]  $\nu'_p(L)=\min \{ r  : \vert \{s\in SC_p: p^{r-1}s\ra L_p \} \vert \ge \vert SC_p\vert -1\}$.
\end{itemize}
Furthermore, if $M$ is a $\z_p$-sublattice of $L$ with rank greater than or equal to $3$, then we have
\begin{itemize} 
\item [(iii)] $\nu_p(L)\le \nu_p(M)$ and $\nu'_p(L)\le \nu'_p(M)$.
\end{itemize}
\end{lem}

\begin{proof}
Since everything is quite clear,  the proofs are left as an exercise to the reader. 
\end{proof}


\begin{lem} \label{newt}  Let $L=L_1\perp \langle p^e\epsilon\rangle$ be a $\z_p$-lattice of rank $4$ such that the scale of the last Jordan component of the ternary sublattice $L_1$ contains $p^e\z_p$, where $e$ is a nonnegative integer and $\epsilon \in \z_p^{\times}$.  Then $\nu_p(L) \le e+1$. Furthermore, if   $L_1$ is anisotropic, then $e-2\delta_{2,p} \le \nu_p(L) \le e+1$.  Here $\delta$ is the Kronecker delta.  
\end{lem}

\begin{proof}  From the assumption, $L$ contains either $p^e\z_p$-modular quaternary $\z_p$-sub-lattice or $\ell_1\perp \ell_2$, where $\ell_1$ is either $0$ or a $p^{e-1}\z_p$-modular $\z_p$-lattice and $\ell_2$ is  a nonzero $p^e\z_p$-modular $\z_p$-lattice. Hence by the result of \cite{OM2}, $L$ represents all $p$-adic integers in $p^e\z_p$. Therefore we have $\nu_p(L) \le e+1$. 

Now, assume that $L_1$ is anisotropic. Then $\q L_1$ is not universal, that is, there is a $p$-adic rational number that is not represented by $L_1$ over $\q_p$. Hence $L$ does not represent at least one element in  $p^{e-2-2\delta_{2,p}}\z_p^{\times} \cup p^{e-1-2\delta_{2,p}}\z_p^{\times}$. Then lemma follows from this.    
\end{proof}

\begin{rmk} \label{rmkanynu}
From the definition, we have $1\le \nu'_p(L)\le \nu_p(L)$ for any prime $p$ and any $\z_p$-lattice $L$ with rank greater than or equal to $3$.
\end{rmk}

\begin{rmk} \label{rmkprime}
By Lemma 3.1 of \cite{CEO}, there is a constant $C$ such that $\nu_p(L)=1$ for any primitive regular $\z$-lattice $L$ of rank greater than or equal to $4$ and  any prime $p$ greater than $C$. 
\end{rmk}


\begin{lem} \label{lemnuprime}
For each prime $p$, there is a constant $\nu'_p$ such that $\nu'_p(L)\le \nu'_p$ for any primitive regular $\z$-lattice $L$ of rank greater than or equal to $4$.
\end{lem}

\begin{proof}
By Corollary 3.2 of \cite{CEO}, there is a finite set of ternary $\z$-lattices 
$$
S=\{ K(1),K(2),\dots,K(r)\}
$$
such that  any primitive regular $\z$-lattice with rank greater than or equal to $4$ represents at least one $K(i) \in S$ for some $i$ with $1\le i\le r$.  
For each prime $p$, we take
$$
C_p=\max \{ \nu'_p(K(i)) : 1\le i\le r\}.
$$
Let $L$ be any regular $\z$-lattice of rank greater than or equal to $4$.
Then there is an integer $j\in \{1,2,\dots,r\}$ such that $K(j)\ra L$.
Then we have, by Lemma \ref{lemnunuprime}(iii), 
$$
\nu'_p(L)\le \nu'_p(K(j))\le C_p.
$$
This completes the proof.
\end{proof}

\begin{lem} \label{lemnuodd}
Let $p$ an odd prime and let $L$ be a $\z_p$-lattice of rank $k$ with $k\ge 4$ such that $L \simeq \langle \epsilon_1,p^{e_2}\epsilon_2,p^{e_3}\epsilon_3,\dots,p^{e_k}\epsilon_k\rangle$, where $0\le e_2\le e_3\le \cdots \le e_k$ and $\epsilon_i\in \z_p^{\times}$.
If $\nu_p(\langle \epsilon_1,p^{e_2}\epsilon_2,p^{e_3}\epsilon_3\rangle)<\infty$, then $\nu_p(L)\le \nu'_p(L)+1$.
\end{lem}

\begin{proof}
Assume that $\nu_p(\langle \epsilon_1,p^{e_2}\epsilon_2,p^{e_3}\epsilon_3\rangle)<\infty$.
Then one may easily check by using Theorem 1 of  \cite{OM2} that 
$$
(\nu'_p(L),\nu_p(L))=\begin{cases}(e_2+1,e_2+1)&\text{if}\ -p^{e_2}\epsilon_1\epsilon_2\in \qpds ,\\
(e_3+1,e_3+1)&\text{if}\ -p^{e_2}\epsilon_1\epsilon_2\not\in \qpds \ \text{and}\ e_2\equiv 0\Mod 2,\\
(e_3,e_3+1)&\text{if}\ -p^{e_2}\epsilon_1\epsilon_2\not\in \qpds \ \text{and}\ e_2\equiv 1\Mod 2.
\end{cases}
$$
The lemma follows directly from this.
\end{proof}

Let $L$ be a primitive $\z_2$-lattice with rank greater than or equal to $4$.
We say {\it $L$ is of type A} if there is a Jordan decomposition
$$
L=\ell_1 \perp \ell_2 \perp \cdots \perp \ell_k \ \ \text{with} \ \  \z_2=\mathfrak s(\ell_1) \supsetneq \mathfrak s(\ell_2) \supsetneq \dots \supsetneq \mathfrak s(\ell_k),
$$
and a positive integer $t$  such that
\begin{equation} \label{atype}
\text{$\ell_1\perp \ell_2\perp \cdots \perp \ell_{t-1}$ is an anisotropic ternary $\z_2$-lattice and}  \ \  
\mathfrak{n}\ell_t \subseteq 8\mathfrak{s}\ell_{t-1}.
\end{equation}  
Note that the former condition of \eqref{atype} depends on a Jordan decomposition of the $\z_2$-lattice $L$, in general.  For example, if $L=\langle 1,3,4\rangle$, then the first Jordan component $\ell_1=\langle1,3\rangle$ is anisotropic. However, since $L\simeq \langle 1,7,20\rangle$, the first Jordan component $\ell_1=\langle 1,7\rangle$ in this decomposition is isotropic.   Note that the latter condition in \eqref{atype}  guarantees that the former condition is independent of a Jordan decomposition of $L$.  

\begin{lem} \label{lemnu2}
Let $L$ be a primitive $\z_2$-lattice with rank greater than or equal to $4$.
If $\nu'_2(L)+3<\nu_2(L)$, then $L$ is of type A.
\end{lem}

\begin{proof}
First of all, note that
\begin{equation} \label{eqlemnu2}
\nu_2(L)\le 4\le \nu'_2(L)+3,\ \ \text{if}\ \ 8\z_2\subset Q(L).
\end{equation}
Therefore we may further assume that $8\z_2\not\subset Q(L)$. From now on, we determine all possible Jordan decompositions of the $\z_2$-lattice  $L$ which satisfy the condition given in the lemma. 

First, assume that
$$
L\simeq \ani \perp 2^{e_3}\both \perp \ell,
$$
where $0\le e_3$, $\both=\hyper$ or $\ani$, and $\mathfrak{n}\ell \subseteq 2^{e_3+1}\z_2$.
Then one may easily check that
$$
Q(L)=\begin{cases}\{ \gamma \in \z_2 : \ord_2(\gamma)\equiv 1\Mod 2 \ \text{or}\ \ord_2(\gamma)\ge e_3\}&\text{if}\ \ e_3\equiv 1\Mod 2,\\
\{ \gamma \in \z_2 : \ord_2(\gamma)\equiv 1\Mod 2 \ \text{or}\ \ord_2(\gamma)\ge e_3+1\}&\text{if}\ \ e_3\equiv 0\Mod 2 .\end{cases}
$$
Hence we have $\nu_2(L)=\nu'_2(L)\in \{ e_3+1,e_3+2\}$ in this case.

Assume that
$$
L\simeq \langle \epsilon_1,2^{e_2}\epsilon_2\rangle \perp 2^{e_3}\both \perp \ell,
$$
where $\epsilon_i\in \z_2^{\times}$, $0\le e_2< e_3$, $\both=\hyper$ or $\ani$, and $\mathfrak{n}\ell \subseteq 2^{e_3+1}\z_2$.
One may easily check that $8\z_2 \subset Q(L)$ when $e_3=1$.
Thus we may assume that $e_3$ is greater than $1$.  If $2^{e_2}\epsilon_1\epsilon_2\in 7\qtwods$, then one may easily check that
$$
\vert \{ s\in \{ 1,3,5,7 \} : 2^{e_2+1}s\ra L\} \vert =0 \  \ \text{and} \ \ 
2^{e_2+2} \z_2 \subset Q(\langle \epsilon_1,2^{e_2}\epsilon_2\rangle) \subset Q(L).
$$
Thus we have $\nu'_2(L)=\nu_2(L)=e_2+3$.
If $2^{e_2}\epsilon_1\epsilon_2\in 3\qtwods$, then one may easily check that
$$
\vert \{ s\in \{1,3,5,7\} : 2^{e_3}s\ra L\} \vert =0 \ \ \text{and} \ \ 2^{e_3+1}\z_2\subset Q(L)
$$
when $e_3\equiv 1\Mod 2$, and
$$
\vert \{ s\in \{1,3,5,7\} : 2^{e_3-1}s\ra L\} \vert =0 \ \ \text{and} \ \ 2^{e_3}\z_2\subset Q(L)
$$
when $e_3\equiv 0\Mod 2$.
Hence we have $\nu_2(L)=\nu'_2(L)\in \{ e_3+1,e_3+2\}$.  Finally, if $2^{e_2}\epsilon_1\epsilon_2\not\in 3\qtwods \cup 7\qtwods$, then one may easily check that
$$
\vert \{ s\in \{ 1,3,5,7 \} : 2^{e_3-2}s\ra L\} \vert \le 2.
$$
From this follows that $e_3\le \nu'_2(L)$. To show that $2^{e_3}\hyper \ra L$,
we may assume that $\both=\ani$.  Since $2^{e_3+1}\epsilon \ra \langle \epsilon_1,2^{e_2}\epsilon_2 \rangle$ for some $\epsilon \in \z_2^{\times}$, we have
$$
\langle 5\cdot 2^{e_3+1}\epsilon \rangle \perp 2^{e_3}\hyper \simeq \langle 2^{e_3+1}\epsilon \rangle \perp 2^{e_3}\ani \ra L.
$$
Hence we have $2^{e_3+1}\z_2 \subset Q(L)$, and  $\nu_2(L)\le e_3+2$.
Since $e_3\le \nu'_2(L)$, we have
$$
\nu_2(L)\le e_3+2\le \nu'_2(L)+2.
$$

Assume that
$$
L\simeq \langle \epsilon \rangle \perp 2^{e_2}\hyper \perp \ell,
$$
where $\epsilon \in \z_2^{\times}$, $1\le e_2$, and $\mathfrak{n}\ell \subseteq 2^{e_2+1}\z_2$.
Clearly, we have $8\z_2 \subset Q(L)$ if $e_2$ is $1$ or $2$. 
Thus we may assume that $3\le e_2$.
If $e_2\equiv 1\Mod 2$,  then one may easily check that
$$
\vert \{ s\in \{1,3,5,7\} : 2^{e_2}s\ra L\} \vert =0\ \ \text{and}\ \ 2^{e_2+1}\z_2\subset Q(L),
$$
and if $e_2\equiv 0\Mod 2$, then 
$$
\vert \{ s\in \{1,3,5,7\} : 2^{e_2-1}s\ra L\} \vert =0\ \ \text{and}\ \ 2^{e_2}\z_2\subset Q(L).
$$
Hence we have $\nu_2(L)=\nu'_2(L)\in \{ e_2+1,e_2+2\}$.

Asssume that
$$
L\simeq \langle \epsilon \rangle \perp 2^{e_2}\ani \perp \ell,
$$
where $\epsilon \in \z_2^{\times}$, $1\le e_2$, and $\mathfrak{n}\ell \subseteq 2^{e_2+1}\z_2$.  First, assume that $e_2$ is odd. 
Then one may easily check that
$$
\vert \{ s\in \{1,3,5,7\} : 2^{e_2}s\ra L\} \vert =0\ \ \text{and}\ \ 2^{e_2+1}\z_2\subset Q(L).
$$
Thus we have $\nu_2(L)=\nu'_2(L)=e_2+2$. Now, assume that $e_2$ is even. Then the ternary $\z_2$-lattice $\langle \epsilon \rangle \perp 2^{e_2}\ani$ is anisotropic. 
Suppose that $L$ is not of type A, that is, $\ord_2(\mathfrak{n}\ell) \in \{e_2+1,e_2+2\}$.
Let $e_3=\ord_2(\mathfrak{n}\ell)$ and let $\delta$ be an element in $\z_2^{\times}$ satisfying $2^{e_3}\delta \ra \ell$.
Since
$$
\vert \{ s\in \{1,3,5,7\} : 2^{e_2-1}s\ra L\} \vert =0,
$$
we have $e_2+1\le \nu'_2(L)$. Since  $1\le e_3-e_2\le 2$, we have $Q(\langle \epsilon \rangle \perp \ani \perp \langle 2^{e_3-e_2}\delta \rangle)=\z_2$. Furthermore, since  $L$ contains a sublattice
$$
2^{e_2}(\langle \epsilon \rangle \perp \ani \perp \langle 2^{e_3-e_2}\delta \rangle),
$$
we have $2^{e_2}\z_2 \subset Q(L)$ and  $\nu_2(L)\le e_2+1$.
This implies that
$$
e_2+1\le \nu'_2(L)\le \nu_2(L)\le e_2+1,
$$
which is a contradiction.

Assume that 
$$
L\simeq \ani \perp \langle 2^{e_2}\epsilon \rangle \perp \ell,
$$
where $\epsilon \in \z_2^{\times}$, $1\le e_2$, and $\mathfrak{s}\ell \subseteq 2^{e_2}\z_2$.
If $e_2$ is odd, then one may easily check that
$$
\vert \{ s\in \{1,3,5,7\} : 2^{e_2-1}s\ra L\} \vert =0\ \ \text{and}\ \ 2^{e_2}\z_2\subset Q(L).
$$
Thus we have $\nu_2(L)=\nu'_2(L)=e_2+1$. Assume that $e_2$ is even. Then the ternary $\z_2$-lattice $\ani \perp \langle 2^{e_2}\epsilon \rangle$ is anisotropic.  Suppose that $L$ is not of type A, that is, 
$e_2\le \ord_2(\mathfrak{s}\ell) \le \ord_2(\mathfrak{n}\ell) \le e_2+2$.
Let $e_3=\ord_2(\mathfrak{n}\ell)$ and let $\delta$ be an element in $\z_2^{\times}$ satisfying $2^{e_3}\delta \ra \ell$.
Since
$$
\vert \{ s\in \{1,3,5,7\} : 2^{e_2-2}s\ra L\} \vert =0,
$$
we have $e_2\le \nu'_2(L)$. Since $2 \le e_3-e_2+2\le 4$, we have $Q(\ani \perp \langle 4\epsilon \rangle \perp \langle 2^{e_3-e_2+2}\delta \rangle)=2\z_2$. 
Furthermore, since $L$ contains a sublattice isometric to
$$
2^{e_2-2}( \ani \perp \langle 4\epsilon \rangle \perp \langle 2^{e_3-e_2+2}\delta \rangle),
$$
we have  $2^{e_2-1}\z_2 \subset Q(L)$ and thus $e_2\le \nu'_2(L)\le\nu_2(L)\le e_2$.
This is a contradiction.

Assume  that
$$
L\simeq \langle \epsilon_1,2^{e_2}\epsilon_2,2^{e_3}\epsilon_3\rangle \perp \ell,
$$
where $\epsilon_i\in \z_2^{\times}$, $0\le e_2\le e_3$, $\nu_2(\langle \epsilon_1,2^{e_2}\epsilon_2,2^{e_3}\epsilon_3\rangle)<\infty$, and $\mathfrak{s}\ell \subseteq 2^{e_3}\z_2$.
First, assume  that $2^{e_2}\epsilon_1\epsilon_2\in 7\qtwods$.
Then one may easily check that
$$
\vert \{ s\in \{1,3,5,7\} : 2^{e_2-1}s\ra L\} \vert=0\ \ \text{and}\ \ 2^{e_2+2}\z_2\subset Q(L).
$$
Hence we have $e_2+1\le \nu'_2(L)$ and $\nu_2(L)\le e_2+3$.
From this follows that $\nu_2(L)\le \nu'_2(L)+2$.
Thus we may further assume that $2^{e_2}\epsilon_1\epsilon_2\not\in 7\qtwods$.  Assume that $2^{e_2}\epsilon_1\epsilon_2\in 3\qtwods$.
Since we are assuming that the ternary sublattice $\langle \epsilon_1,2^{e_2}\epsilon_2,2^{e_3}\epsilon_3\rangle$ of $L$ is isotropic, we have $e_3\equiv 0\Mod 2$.   Now, one may easily check that
$$
\vert \{ s\in \{1,3,5,7\} : 2^{e_3-1}s\ra L\} \vert=0.
$$
 If $2^{e_2}\epsilon_1\epsilon_2\not\in 3\qtwods \cup 7\qtwods$, then we have
$$
\vert \{ s\in \{1,3,5,7\} : 2^{e_3-3}s\ra L\} \vert \le 2.
$$
From these follows that $e_3-1\le \nu'_2(L)$.
Note that $\langle \epsilon_1,2^{e_2}\epsilon_2,2^{e_3}\epsilon_3\rangle$ contains a sublattice isometric to either
$2^{e_3-1}\langle \epsilon_1,2^{\delta_2}\epsilon_2,2^{\delta_3}\epsilon_3\rangle$ or  $2^{e_3}\langle \epsilon_1,\epsilon_2,\epsilon_3\rangle$,
where $\delta_2, \delta_3 \in \{0,1\}$.
By using  Theorem 3 of \cite{OM2}, one may easily check that
$$
2\z_2\subset Q(\langle \epsilon_1,2^{\delta_2}\epsilon_2,2^{\delta_3}\epsilon_3\rangle )\ \ \text{and}\ \ Q(\langle \epsilon_1,\epsilon_2,\epsilon_3\rangle)=\z_2.
$$
Hence we have  $2^{e_3}\z_2\subset Q(L)$ and thus  $\nu_2(L)\le e_3+1$.
Consequently, we have
$$
\nu_2(L)\le e_3+1\le \nu'_2(L)+2.
$$

Finally, assume that
$$
L\simeq \langle \epsilon_1,2^{e_2}\epsilon_2,2^{e_3}\epsilon_3\rangle \perp \ell,
$$
where $\epsilon_i\in \z_2^{\times}$, $0\le e_2\le e_3$, $\nu_2(\langle \epsilon_1,2^{e_2}\epsilon_2,2^{e_3}\epsilon_3\rangle)=\infty$. Suppose that $L$ is not of type A.  Then we may assume that
$$
e_3\le \ord_2(\mathfrak{s}\ell) \le \ord_2(\mathfrak{n}\ell)\le e_3+2.
$$
Let $M=\langle \epsilon_1,2^{e_2}\epsilon_2,2^{e_3}\epsilon_3\rangle$ and let $\gamma$ be the unique element in $\{ 2^{\nu'_2(M)-1}s : s\in SC_2\}$ satisfying $\gamma \nra M$.
Let $e_4=\ord_2(\mathfrak{n}\ell)$ and let $\delta' \in \z_2^{\times}$ be an element such that  $2^{e_4}\delta' \ra \ell$.
One may easily check that
$$
\begin{cases} 
\vert \{ s\in \{1,3,5,7\} : 2^{e_3-2}s \ra M \} \vert =0 \quad &\text{if $2^{e_2}\epsilon_1\epsilon_2\in 3\qtwods$ and $e_3\equiv 1\Mod 2$},\\ 
\vert \{ s\in \{1,3,5,7\} : 2^{e_3-1}s \ra M \} \vert =0 \quad &\text{if $2^{e_2}\epsilon_1\epsilon_2\in 3\qtwods$ and  $e_3\equiv 0\Mod 2$},  \\
\vert \{ s\in \{1,3,5,7\} : 2^{e_3-3}s \ra M \} \vert \le 2 \quad &\text{if $2^{e_2}\epsilon_1\epsilon_2\not\in 3\qtwods \cup 7\qtwods$}. \end{cases}
$$
From this follows that $e_3-1\le \nu'_2(M)$. On the other hand, $M$ contains a sublattice isometric to either
$2^{e_3-1}\langle \epsilon_1,2^{\delta_2}\epsilon_2,2^{\delta_3}\epsilon_3\rangle$ or  $2^{e_3}\langle \epsilon_1,\epsilon_2,\epsilon_3\rangle$,
where $\delta_i\in \{0,1\}$.
Now, by using  Theorem 3 of \cite{OM2}, one may easily show that
$\nu'_2(\langle \epsilon_1,2^{\delta_2}\epsilon_2,2^{\delta_3}\epsilon_3\rangle)=1$.
From this follows that $\nu'_2(M)\le e_3+1$.
Hence we have
\begin{equation} \label{eqlemnu22}
e_3-1\le \nu'_2(M)\le e_3+1.
\end{equation}
Now, assume that $\nu'_2(M)+2\le e_4$.
Since
$$
\vert \{ s\in SC_2 : 2^{\nu'_2(M)-2}s\ra L \} \vert =\vert \{ s\in SC_2 : 2^{\nu'_2(M)-2}s\ra M \} \vert \le 6,
$$
we have $\nu'_2(M)\le \nu'_2(L)$.
Thus $\nu'_2(M)=\nu'_2(L)$ by Lemma \ref{lemnunuprime}(iii).
On the other hand, since $e_3-1\le \nu'_2(M)$, we have
$$
\nu'_2(M)+2\le e_4\le e_3+2\le \nu'_2(M)+3.
$$
From this and the fact that
$$
\{ 3\cdot 4\gamma, 5\cdot 4\gamma, 7\cdot 4\gamma \} \subset Q(M),
$$
one may easily show that either $4\gamma-2^{e_4}\delta'\in Q(M)$ or $4\gamma-2^{e_4+2}\delta'\in Q(M)$.
Hence we have $4\gamma \in Q(L)$ and it follows that
$$
\nu_2(L)\le \nu'_2(M)+2=\nu'_2(L)+2.
$$
Finally, assume  that $e_4\le \nu'_2(M)+1$.
Then there is a nonnegative integer $g$ such that
$$
\ord_2(4^g 2^{e_4}\delta')\in \{ \nu'_2(M),\nu'_2(M)+1\}.
$$
Since
$\gamma-4^g 2^{e_4}\delta' \ra M\ \ \text{or}\ \ \gamma-4^{g+1} 2^{e_4}\delta' \ra M$,
 we have $\gamma \ra L$.
Hence $\nu_2(L)\le \nu'_2(M)$.
Furthermore, since $\nu'_2(M)-4\le e_3-3\le e_4-3$ by \eqref{eqlemnu22}, it follows that
$$
\vert \{ s\in SC_2 : 2^{\nu'_2(M)-5}s\ra L \} \vert =\vert \{ s\in SC_2 : 2^{\nu'_2(M)-5}s\ra M \} \vert \le 6.
$$
Thus we have $\nu'_2(M)-3\le \nu'_2(L)$.
From the inequalities
$$
\nu'_2(M)-3\le \nu'_2(L)\le \nu_2(L)\le \nu'_2(M),
$$
we have $\nu_2(L)\le \nu'_2(L)+3$.  This completes the proof.
\end{proof}

\begin{lem} \label{lemanisoodd}
Let $p$ an odd prime and let $L$ be a $\z_p$-lattice of rank $k$ with $k\ge 4$ such that $L\simeq \langle \epsilon_1,p^{e_2}\epsilon_2,p^{e_3}\epsilon_3,\dots,p^{e_k}\epsilon_k\rangle$, where $0\le e_2\le e_3\le \cdots \le e_k$ and $\epsilon_i\in \z_p^{\times}$.
If $\nu_p(\langle \epsilon_1,p^{e_2}\epsilon_2,p^{e_3}\epsilon_3\rangle)=\infty$, then $B(\mathbf{x},\mathbf{y})\in {p^{\nu_p(L)-1}}\z_p$ for any vectors $\mathbf{x},\mathbf{y}\in L$ with $\ord_p(Q(\mathbf{x}))=\ord_p(Q(\mathbf{y}))=\nu_p(L)$.
\end{lem}

\begin{proof}
By Lemma \ref{newt}, we have $e_4 \le \nu_p(L)\le e_4+1$.
Let $\{ \mathbf{v}_1,\mathbf{v}_2,\dots,\mathbf{v}_k\}$ be a $\z_p$-basis for $L$ such that
$$
L=\z_p\mathbf{v}_1+\z_p\mathbf{v}_2+\cdots+\z_p\mathbf{v}_k \simeq \langle \epsilon_1,p^{e_2}\epsilon_2,p^{e_3}\epsilon_3,\dots,p^{e_k}\epsilon_k\rangle.
$$
Suppose that  there are vectors $\mathbf{x}, \mathbf{y}\in L$ with $\ord_p(Q(\mathbf{x}))=\ord_p(Q(\mathbf{y}))=\nu_p(L)$ such that $B(\mathbf{x},\mathbf{y})\not\in p^{\nu_p(L)-1}\z_p$.
Then we may write
$$
Q(\mathbf{x})=p^{\nu_p(L)}a,\ \ Q(\mathbf{y})=p^{\nu_p(L)}b,\ \ \text{and} \ \ B(\mathbf{x},\mathbf{y})=p^rc,
$$
where $a,b,c\in \z_p^{\times}$ and $0\le r\le \nu_p(L)-2$.
Since $\nu_p(L)\in \{e_4,e_4+1\}$, we have
$$
r\le \nu_p(L)-2\le e_4-1.
$$
Let $M=\z_p \mathbf{x}+\z_p \mathbf{y}$ be the $\z_p$-sublattice of $L$.
Since $r\le e_4-1$, we have 
$$
\q_pM\ra \q_p\mathbf{v}_1+\q_p\mathbf{v}_2+\q_p\mathbf{v}_3
$$
by Theorem 1 of \cite{OM2}.
Since  $M\simeq p^r\langle 1,-1\rangle$ is isotropic, it can not be represented by the anisotropic quadratic $\q_p$-space $\q_p\mathbf{v}_1+\q_p\mathbf{v}_2+\q_p\mathbf{v}_3$.
This completes the proof.
\end{proof}

\begin{lem} \label{lemaniso2}
Let $L$ be a primitive $\z_2$-lattice with rank greater than or equal to $4$.
If $L$ is of type A, then we have $B(\mathbf{x},\mathbf{y})\in 2^{\nu_2(L)-2}\z_2$ for any vectors $\mathbf{x},\mathbf{y}\in L$ with $\ord_2(Q(\mathbf{x}))=\ord_2(Q(\mathbf{y}))=\nu_2(L)$.
\end{lem}

\begin{proof}
We may  assume that $\nu_2(L)\ge 3$, since otherwise there is nothing to prove.
Since we are assuming that $L$ is of type A, we may write $L=\ell'\perp \ell$ such that $\ell'$ is an anisotropic ternary $\z_2$-lattice and $\mathfrak n(\ell) \subseteq 8\mathfrak a$, where $\mathfrak a$ is the scale of a last Jordan component of $\ell'$ by \eqref{atype}.  Let $e_4=\ord_2(\mathfrak{n}\ell)$ and let $\delta \in \z_2^{\times}$ be an element such that $\langle 2^{e_4}\delta\rangle \ra \ell$.   Since
$$
\ell' \perp \langle 2^{e_4}\delta\rangle \ra L,
$$
we have $\nu_2(L)\le e_4+1$ by Lemmas \ref{lemnunuprime}(iii) and \ref{newt}.

Suppose there are vectors $\mathbf{x},\mathbf{y}\in L$ with $\ord_2(Q(\mathbf{x}))=\ord_2(Q(\mathbf{y}))=\nu_2(L)$ such that $B(\mathbf{x},\mathbf{y})\not\in 2^{\nu_2(L)-2}\z_2$.
Then we may write
$$
Q(\mathbf{x})=2^{\nu_2(L)}a,\ \ Q(\mathbf{y})=2^{\nu_2(L)}b,\ \ B(\mathbf{x},\mathbf{y})=2^rc,
$$
where $a,b,c\in \z_2^{\times}$ and $0\le r\le \nu_2(L)-3$.
Put 
$$
M=\z_2\mathbf{x}+\z_2\mathbf{y}\simeq \begin{pmatrix}2^{\nu_2(L)}a&2^rc\\2^rc&2^{\nu_2(L)}b\end{pmatrix}.
$$
Since $r\le \nu_2(L)-3$, it follows that
$$
M\simeq 2^r\begin{pmatrix}2^{\nu_2(L)-r}a&c\\c&2^{\nu_2(L)-r}b\end{pmatrix} \simeq 2^r\hyper.
$$
Thus we have $2^r\hyper \ra L$.
Hence, there are vectors $\mathbf{z}',\mathbf{w}'\in \ell'$ and $\mathbf{z},\mathbf{w}\in \ell$ such that 
$$
\z_2(\mathbf{z}'+\mathbf{z})+\z_2(\mathbf{w}'+\mathbf{w})\simeq 2^r\begin{pmatrix}0&1\\1&0\end{pmatrix}.
$$
From this and the fact that $\mathfrak{s}\ell \subseteq 2^{e_4-1}\z_2$, one may easily show that
$$
\begin{array}{l}
Q(\mathbf{z}')=-Q(\mathbf{z})\equiv 0\Mod {2^{e_4}},\\[0.3em]
Q(\mathbf{w}')=-Q(\mathbf{w})\equiv 0\Mod {2^{e_4}},\\[0.3em]
B(\mathbf{z}',\mathbf{w}')=2^r-B(\mathbf{z},\mathbf{w})\equiv 2^r\ \ \text{or}\ \ 2^r+2^{e_4-1} \Mod {2^{e_4}}.
\end{array}
$$
Put $K=\z_2\mathbf{z}'+\z_2\mathbf{w}'\subset \ell'$.
Since $\nu_2(L)\le e_4+1$, we have
$$
r\le \nu_2(L)-3\le e_4-2.
$$
From this, one may easily see that
$$
K\simeq \begin{pmatrix}Q(\mathbf{z}')&B(\mathbf{z}',\mathbf{w}')\\B(\mathbf{z}',\mathbf{w}')&Q(\mathbf{w}')\end{pmatrix}
\simeq 2^r\hyper.
$$
Therefore we have $2^{r+1}\z_2 \subset Q(K) \subset Q(\ell')$.  Since we are assuming that $\nu_2(\ell')=\infty$, this is a contradiction. This completes the proof.
\end{proof}

\section{The rank of new regular $\z$-lattices}

Let $A$ be the set of all primes $p$ such that there is a primitive regular $\z$-lattice $L$ with rank greater than or equal to $4$ satisfying $\nu_p(L)>1$.
Then $A$ is a finite set, as stated in Remark \ref{rmkprime}, and thus we may write
$$
A=\{p_1,p_2,\dots,p_a\}.
$$
For each $\eta=(\eta_{p_1},\eta_{p_2},\dots,\eta_{p_a}) \in \prod_{1\le i\le a} \z_{p_i}$, define a set $S_{\eta}$ by 
$$
S_{\eta}=\left\{ n\in \n : \eta_{p_i}^{-1}n\in (\z_{p_i}^{\times})^2,\ i=1,2,\dots,a \right\}.
$$
For each $i=1,2,\dots,a$, we choose a positive integer $\nu'_{p_i}$ satisfying the condition given  in  Lemma \ref{lemnuprime}.
Define a set $B_{p_i} \subset \z_{p_i}$ by
$$
B_{p_i}=\begin{cases}\left\{ 2^a b : 0\le a\le \nu'_2+3,\ b=1,3,5,7 \right\}&\text{if}\ \ p_i=2, \\
\left\{ {p_i}^ab : 0\le a\le \nu'_{p_i}+3,\ b=1,\Delta_{p_i} \right\}&\text{if}\ \ p_i\in A-\{2\}, \end{cases}
$$
and a finite subset $B=\prod_{1\le i\le a} B_{p_i} \subset \prod_{1\le i\le a} \z_{p_i}$.
For each $\eta \in B$, let $T_{\eta}$ be a finite $S_{\eta}$-universality criterion set. Such a set  $T_{\eta}$ exists for any $\eta \in B$  by Theorem 3.3 of  \cite{KKO}.  Finally, we define $t_{\eta}=\vert T_{\eta}\vert$. 

Now, we are ready to prove our main theorem.

\begin{thm} \label{thmmain}
The rank of a primitive new regular $\z$-lattice is less than or equal to $\sum_{\eta \in B}t_{\eta}$.
\end{thm}

\begin{proof}
Let $L$ be a primitive new regular $\z$-lattice with rank greater than or equal to $4$.
It suffices to show that there is a $\z$-sublattice $\widetilde{L}$ of $L$ with $\text{rank}(\widetilde{L})\le \sum_{\eta \in B}t_{\eta}$ such that $Q(\widetilde{L})=Q(L)$.
Let
$$
E=\{ 1\le i\le a : \nu'_{p_i}(L)+3<\nu_{p_i}(L)\}.
$$
For each $i\in E$, we define $\xi_i={p_i}^2 c_i$, where $c_i\in SC_{p_i}$ such that $\nu_{p_i,c_i}(L)=\nu_{p_i}(L)$.
Let
$$
H(L)=\left\{ \left( \eta_{p_1,s_1}(L),\eta_{p_2,s_2}(L),\dots,\eta_{p_a,s_a}(L)\right)\in \prod_{1\le i\le a} \z_{p_i} : s_i\in SC_{p_i} \right\}.
$$
Since we are assuming that $L$ is regular, we may easily check that 
$$
Q(L)-\{0\}=\displaystyle \left\lbrace u^2v : u \in \mathbb N, v \in \bigcup_{\eta \in H(L)} S_{\eta}\right\rbrace.
$$
For each
$$
\eta=(\eta_{p_1,h_1}(L),\eta_{p_2,h_2}(L),\dots,\eta_{p_a,h_a}(L))\in H(L),
$$
 where $h_i\in SC_{p_i}$ for $i=1,2,\dots,a$,
we define
$$
E_{\eta}=\{ i\in E : \nu_{p_i,h_i}(L)=\nu_{p_i}(L)\}.
$$
Note that for any $j \not \in E_{\eta}$, either $\nu_{p_j,h_j}(L) \ne \nu_{p_j}(L)$ or $\nu_{p_j}(L)\le \nu'_{p_j}(L)+3$. Hence we have
$$
\ord_{p_j}(\eta_{p_j,h_j}(L)) \le \nu'_{p_j}(L) \le \nu'_{p_j}+3.
$$
For each $j\in E_{\eta}$,  let $\kappa_j$ be an integer such that $\eta_{p_j,h_j}(L)={p_j}^{2\kappa_j}\xi_j$.
Since $2 \le \ord_{p_j}(\xi_j)\le 3$, we have
\begin{equation} \label{eqmain2}
2 \le \nu'_{p_j}(L)+1 \le  \nu_{p_j}(L)-3 \le  2\kappa_j=\nu_{p_j,h_j}(L)-\ord_{p_j}(\xi_j)\le \nu_{p_j}(L)-2.
\end{equation}
Hence $\kappa_j$ is a positive integer which depends on the  $\z$-lattice $L$.
 For each $i=1,2,\dots,a$, we define
$$
\eta'_i=\begin{cases}\eta_{p_i,h_i}(L)&\text{if}\ \ i\not\in E_{\eta},\\
\xi_i&\text{otherwise},\end{cases}
$$
and $\eta'=(\eta'_1,\eta'_2,\dots,\eta'_a)$.
Clearly, the map $\eta \mapsto \eta'$ from $H(L)$ to $B$ is a well-defined injective map. From the definition, one may easily show that 
$S_{\eta}=\{ m^2n : n\in S_{\eta'} \}$, where $m=m(L)=\prod_{j\in E_{\eta}}{p_j}^{\kappa_j}$.

Let $T_{\eta'}=\{n_1,n_2,\dots,n_{t_{\eta'}}\}$ be a finite $S_{\eta'}$-universality criterion set.  
For each $l=1,2,\dots,t_{\eta'}$, since $m^2n_l\in S_{\eta}\subset Q(L)$, there is a vector $\mathbf{x}_l\in L$ such that
$Q(\mathbf{x}_l)=m^2n_l$.
For $\mathbf{y}_l=\displaystyle\frac{1}{m}\cdot \mathbf{x}_l \in \q L$, we define a $\z$-lattice  
$$
L(\eta')=\z \mathbf{y}_1+\z \mathbf{y}_2+\cdots+\z \mathbf{y}_{t_{\eta'}}.
$$
We claim that the $\z$-lattice $L(\eta')$ in the quadratic space $\q L$ is integral.
To prove the claim, it suffices to show that, for any $l, l'$ with  $1\le l<l'\le t_{\eta'}$,
$$
B(\mathbf{x}_{l},\mathbf{x}_{l'})\equiv 0\Mod {p^{2\kappa_j}}\ \ \text{for any $j\in E_{\eta}$}.
$$
However, this follows immediately from $\eqref{eqmain2}$ and Lemmas \ref{lemnuodd}, \ref{lemnu2}, \ref{lemanisoodd}, and \ref{lemaniso2}.
Therefore $L(\eta')$ is an integral $\z$-lattice which represents all integers in $T_{\eta'}$.
Since $T_{\eta'}$ is a finite $S_{\eta'}$-universality criterion set, the $\z$-lattice $L(\eta')$ represents all integers in $S_{\eta'}$, that is, 
$S_{\eta'}\subseteq Q(L(\eta'))$.
Let $L(\eta)$ be the sublattice $\z \mathbf{x}_1+\z \mathbf{x}_2+\cdots+\z \mathbf{x}_{t_{\eta'}}$ of $L$.
Since $S_{\eta}=\{m^2n : n\in S_{\eta'}\}$ as noted above, we have $S_{\eta}\subseteq Q(L(\eta))$.
Consequently, if we define 
$$
\widetilde{L}=\sum_{\eta \in H(L)} L(\eta) \le L,
$$
then $S_{\eta}\subset Q(L(\eta))\subseteq Q(\widetilde{L})$ for any $\eta \in H$. Therefore we have
$$
Q(L)-\{0\}=\left\{ u^2v : u\in \n,\ v\in\bigcup_{\eta \in H(L)}S_{\eta} \right\}\subseteq Q(\widetilde{L}),
$$
which implies that $Q(\widetilde{L})=Q(L)$. Note that
$$
\text{rank}(\widetilde{L})\le \sum_{\eta \in H}t_{\eta'}\le \sum_{\eta' \in B}t_{\eta'}.
$$
This completes the proof.
\end{proof}


\end{document}